\documentclass[preprint]{elsarticle}
\usepackage{amsmath,amsfonts,amsthm,url,color,amssymb}
\usepackage{verbatim}
\usepackage{subcaption}
\usepackage{enumerate}
\usepackage{graphicx}
\usepackage{url}
\usepackage{tikz}
\newtheorem{theorem}{Theorem}[section]
\newtheorem{proposition}[theorem]{Proposition}
\newtheorem{lemma}[theorem]{Lemma}
\newtheorem{corollary}[theorem]{Corollary}
\newtheorem{definition}[theorem]{Definition}
\newtheorem{remark}[theorem]{Remark}
\newtheorem{example}[theorem]{Example}
\newtheorem{problem}{Problem}
\newtheorem{conjecture}{Conjecture}

\newcommand{\I}{\mathcal I}
\newcommand{\F}{\mathbb F}

\newcommand{\G}{\mathcal{G}}

\newcommand{\GL}{\mathrm{GL}}
\newcommand{\PGL}{\mathrm{PGL}}

\newcommand{\ord}{\mathrm{ord}}

\newcommand{\doublespace}

\begin{document}
\begin{frontmatter}

\title{Construction of irreducible polynomials \\
through rational transformations}

\author{Daniel Panario}
\ead{daniel@math.carleton.ca}
\author{Lucas Reis \fnref{fn1}\corref{cor1}}
\ead{lucasreismat@gmail.com}
\author{Qiang Wang}
\ead{wang@math.carleton.ca}
\fntext[fn1]{Permanent address: Departamento de Matem\'{a}tica, Universidade Federal de Minas Gerais, UFMG, Belo Horizonte, MG, 30123-970, Brazil}
\address{School of Mathematics and Statistics, Carleton University, 1125 Colonel By Drive, Ottawa ON (Canada), K1S 5B6}
\cortext[cor1]{Corresponding author}

\begin{abstract}
Let $\F_q$ be the finite field with $q$ elements, where $q$ is a power of a prime. 
We discuss recursive methods for constructing irreducible polynomials over $\F_q$ 
of high degree using rational transformations. In particular, given a divisor 
$D>2$ of $q+1$ and an irreducible polynomial $f\in \F_{q}[x]$ of degree $n$ such 
that $n$ is even or $D\not \equiv 2\pmod 4$, we show how to obtain from $f$ a 
sequence $\{f_i\}_{i\ge 0}$ of irreducible polynomials over $\F_q$ with 
$\deg(f_i)=n\cdot D^{i}$.
\end{abstract}
\end{frontmatter}

\section{Introduction}
Let $q$ be a prime power, $\F_q$ be the finite field with $q$ elements and $\overline{\F}_q$ be its algebraic closure. If $q_0=q^n$, the finite field $\F_{q_0}$ can be constructed via the isomorphism $\F_{q_0}\cong \F_q[x]/\langle f\rangle$, where $f\in \F_q[x]$ is an irreducible polynomial of degree $n$. Though the existence of such an $f$ is known, the efficient construction of irreducible polynomials of a given degree is still an interesting research problem. In many practical situations, the construction of large fields is required; see~\cite{GR} and~\cite{SSS}. In the past few years, this problem has been considered by many authors and, in fact, the techniques employed follow a similar pattern. In general, it is considered $f\in \F_q[x]$ an irreducible polynomial of degree $n$ and a rational function $Q=\frac{g}{h}$ of degree $k>1$. From $f$ and $Q$, we consider the sequence of polynomials $\{f_i\}_{i\ge 0}$ given by $f_0=f$ and, for $i\ge 1$, $f_i^{Q}:=h^{d_{i-1}}f_{i-1}\left(\frac{g}{h}\right)$, where $d_i$ is the degree of $f_i$. The main problem relies on giving conditions on $f$ and $Q$ for which is ensured that each polynomial $f_i$ is irreducible. In other words, we are interested in the following problems:

\begin{enumerate}[(A)]
\item Given $Q$ and $n$, find an irreducible polynomial $f$ of degree $n$ such that $f^Q$ is irreducible.

\item Provided that $f$ and $f^Q$ are irreducible, is it true that $(f^Q)^Q$ is also irreducible?
\end{enumerate}

In this case, this recursive method would give a sequence $\{f_i\}_{i\ge 0}$ of irreducible polynomials whose degrees grow~\emph{exponentially}; in fact, if each $f_i$ is irreducible, then $\deg(f_i)=nk^i$, where $k$ is the degree of $Q$. The quotients $K_i=\F_q[x]/\langle f_i\rangle\cong \F_{q^{nk^i}}$ yield the tower $K_1\subset K_2\subset \cdots$ of finite fields. A general criteria on the irreducibility of compositions $f^Q$ is given in~\cite{C69}. This criteria is, perhaps, the one used in most of the previous articles: see~\cite{M90, U13}, where the rational function $Q$ has degree $2$ and~\cite{AAK}, where $Q$ has degree $p$, the characteristic of $\F_q$. 

In this paper, we discuss Questions (A) and (B) above for a special class of rational functions $Q$, introduced in~\cite{R17}, that come from an action of the group $\PGL_2(\F_q)$ on irreducible polynomials over $\F_q$. For these rational functions, we provide a constructive and a probabilistic solution to Question (A) and partially answer Question (B).
 
The structure of the paper is given as follows. In Section 2 we provide background material that is used along the way. In Section 3 we propose solutions for Question (A). In Section 4, we discuss Question (B) and in Section 5 we propose some problems for future research.

\section{Preliminaries}
In this section, we provide some basic definitions and results that are useful along this paper. Let us  fix some notation. For a prime number $r$ and a positive integer $m$, we define $\nu_r(m)\in \mathbb N$ to be the greatest power of $r$ that divides $m$.  Additionally, for relatively prime integers $a$ and $b$, let $\ord_b(a)$ be the multiplicative order of $a$ modulo $b$, i.e., the least positive integer $d$ such that $a^d\equiv 1\pmod b$. Let
\begin{itemize}
\item $\ord(\alpha):=\min\{d>0\,|\, \alpha^d=1\},$
\item $\I_k:=\{f\in \F_q[x]\,|\, f\,\text{is monic, irreducible and}\, \deg(f)=k\},$
\item $\GL_2(\F_q)\; \text{is the general linear group of order}\; 2,$
\item $\PGL_2(\F_q)\; \text{is the projective general linear group of order}\; 2$.
\end{itemize}
For $A\in \GL_2(\F_q)$, $[A]\in \PGL_2(\F_q)$ denotes the equivalence class of $A$, i.e., $$[A]=\{B\in \GL_2(\F_q)\,|\, B=\lambda \cdot A, \lambda\in \F_q^* \}.$$ Let $\F_{q}[x]$ denote the polynomial ring over $\F_q$ and let $\F_q(x)$ be the field of rational functions over $\F_q$.

\subsection{An action of $\PGL_2(\F_q)$ on the sets $\I_k$}
Given $[A]\in \PGL_2(\F_q)$, with $A=
\left(\begin{matrix}
a&b\\
c&d
\end{matrix}\right)
$, and $f\in \I_k$, we set
$$[A]\circ f=\lambda_{A, f}\cdot (bx+d)^k\cdot f\left(\frac{ax+c}{bx+d}\right),$$
where $\lambda_{A, f}\in \F_q$ is the unique element in $\F_q$ such that $\lambda_{A, f}\cdot (bx+d)^k\cdot f\left(\frac{ax+c}{bx+d}\right)$ is monic.  
As pointed out in \cite{ST}, the group $\PGL_2(\F_q)$ acts on each set $\I_k$ with $k\ge 2$, via the compositions $[A]\circ f$. We set 
$$C_A=\{f\in \F_q[x]\,|\, f\,\text{is monic, irreducible, of degree at least two and}\, [A]\circ f=f\},$$
and, for $k\ge 2$, we set,
$$C_A(k)=C_A\cap \I_k.$$

In the same paper, the authors obtain a characterization of the elements in $C_A$ for any $[A]\in \PGL_2(\F_q)$: these {\em invariant} polynomials appear as the irreducible factors of a special class of polynomials over $\F_q$. They also prove that, if $D=\ord([A])$, then any polynomial $f\in C_A$ is either quadratic or has degree divisible by $D$. In particular, $|C_A(k)|=0$ if $k>2$ and $k$ is not divisible by $D$. Also, for $k=Dm$, they prove that
$$|C_A(Dm)|\approx \frac{\varphi(D)}{Dm}q^m,$$
where $\varphi$ is the Euler function. Here, $a_n\approx b_n$ means $\lim\limits_{n\to\infty}\frac{a_n}{b_n}=1$.

\subsubsection{Some recent results}
For each $[A]\in \PGL_2(\F_q)$ of order $D>2$, there exists a rational function $Q=Q_A$ of degree $D$ with the property that any polynomial of degree $Dm>2$ in $C_A$ is of the form $f^Q$, for some polynomial $f$ of degree $m$, see~\cite{R17}. The main idea relies on considering four special classes of elements in $\PGL_2(\F_q)$. For the elements in these classes that are in reduced form (in the sense of Definition 2.6 of~\cite{R17}), the rational function $Q_A$ is readily given. Namely, the main results in Section 3 of~\cite{R17} can be summarized as follows.

\begin{theorem}\label{thm:sec-6-4}
Let $A\in \GL_2(\F_q)$ such that $A$ has one of the following forms
$$\left(\begin{matrix}
a&0\\
0&1
\end{matrix}\right), \left(\begin{matrix}
1&0\\
1&1
\end{matrix}\right), \left(\begin{matrix}
0&1\\
b&0
\end{matrix}\right), \left(\begin{matrix}
0&c\\
1&1
\end{matrix}\right),$$
and let $D$ be the order of $[A]$ in $\PGL_2(\F_q)$. Then the elements of $C_A$ of degree $Dm>2$ are the monic irreducible polynomials of the form $f^{Q_A}=h_A^{m}\cdot f(Q_A)$, where $f\in \F_q[x]$ has degree $m$ and $Q_A=\frac{g_A}{h_A}\in \F_q(x)$ is given as follows:
\begin{itemize}
\item $Q_A(x)=x^D$ if $A=\left(\begin{matrix}
a&0\\
0&1
\end{matrix}\right)$, and $a\in \F_q^*$ is such that $\ord(a)=D$,

\item $Q_A(x)=x^p-x$ if $A=\left(\begin{matrix}
1&0\\
1&1
\end{matrix}\right)$,

\item $Q_A(x)=\frac{x^2+b}{x}$ if $A=\left(\begin{matrix}
0&1\\
b&0
\end{matrix}\right)$ is such that $x^2-b\in \F_q[x]$ is irreducible,

\item $Q_A(x)=\frac{\theta(x+\theta^q)^D-\theta^q(x+\theta)^D}{(x+\theta)^D-(x+\theta^q)^D}$ if $A=\left(\begin{matrix}
0&c\\
1&1
\end{matrix}\right)$ is such that $x^2-x-c\in \F_q[x]$ is irreducible, and $\theta\in \F_{q^2}\setminus \F_q$ is an eigenvalue of $A$ (i.e., $\theta$ is a root of $x^2-x-c$).
\end{itemize}
\end{theorem}

In particular, this result suggests the construction of irreducible polynomials of high degree.  One may consider the four types of rational functions and discuss the problems (A) and (B) given in the introduction. We observe that only the degree of $Q_A$ matters. In particular, the case $Q_A$ of degree two is covered by the cases $Q_A(x)=x^2$ (in odd characteristic) and $Q_A(x)=x^2+x$ (in characteristic two). We are left to consider three cases. It turns out that the cases $Q_A(x)=x^D$ and $Q_A(x)=x^p-x$ are strongly related to some classical constructions of irreducible polynomials and are dealt  in detail in \cite{thesis}. In fact, quantitative aspects of irreducible polynomials of the form $f(x^D)$ were earlier given in~\cite{C69}. For this reason, in the present text, we focus on the fourth class of rational functions that is given in Theorem~\ref{thm:sec-6-4}.

For convenience, we set $A_c=\left(\begin{matrix}
0&1\\
c&1\end{matrix}\right)$, where $x^2-x-c$ is an irreducible polynomial over $\F_q$. Let $\theta$ and $\theta^q$ be the roots of $x^2-x-c$; these are the eigenvalues of $A_c$. In other words, $A$ is diagonalizable over $\F_{q^2}$ but not over $\F_q$.  The order $D$ of $[A_c]$ is then the minimal positive integer $d$ such that $\theta^d=\theta^{qd}$, i.e., $\theta^{(q-1)d}=1$. As pointed out in \cite{R17}, $D>2$ is always a divisor of $q+1$ and the converse is also true, i.e., for any divisor $D>2$ of $q+1$, there exists $c\in \F_q$ such that $x^2-x-c$ is irreducible over $\F_q$ and $[A_c]$ has order $D$; see Lemma 2.7 and Proposition 2.8 of~\cite{R17}. In the rest of this paper, we always consider $c\in \F_q$ such that $\ord([A_c]):=D>2$ is a divisor of $q+1$.

\begin{definition}\label{def:Q_c}
Set $g_c(x)=\frac{\theta(x+\theta^q)^D-\theta^q(x+\theta)^D}{\theta^q-\theta}$ and $h_c(x)=\frac{(x+\theta)^D-(x+\theta^q)^D}{\theta^q-\theta}$. The rational function $$Q_c(x)=\frac{g_c(x)}{h_c(x)}=\frac{\theta(x+\theta^q)^D-\theta^q(x+\theta)^D}{(x+\theta)^D-(x+\theta^q)^D},$$
is the canonical rational function associated to $A_c$. Also, for $f\in \F_q[x]$ of degree $m$, we set
$$f^{Q_c}=h_c^m\cdot f\left(\frac{g_c}{h_c}\right).$$
\end{definition}

According to Proposition 3.7 and Theorem 3.8 of~\cite{R17}, we have the following result. 

\begin{theorem}\label{main} The following hold.
\begin{enumerate}
\item The polynomials $g_c, h_c$ are relatively prime and their coefficients lie in $\F_q$. 
\item The equality $Q_c\left(\frac{c}{x+1}\right)=Q_c(x)$ holds.
\item The poles of $Q_c$ over $\overline{\F}_q$ are in $\F_{q^2}$.
\item For each positive integer $m$ such that $Dm>2$, the elements of $C_{A_c}(Dm)$ are exactly the monic irreducible polynomials of the form $$f^{Q_c}=h_c^m\cdot f\left(\frac{g_c}{h_c}\right),$$ where $f\in \F_q[x]$ has degree $m$.
\end{enumerate}
\end{theorem}

We know that there exist many irreducible polynomials of degree $Dm$ that are invariant under $[A_c]$. In fact, according to Theorem 4.7 of~\cite{R18},
\begin{equation}\label{enumeration-reis} |C_{A_c}(Dm)|=\frac{\varphi(D)}{Dm}\sum_{d|m\atop \gcd(d, D)=1}(q^{m/d}+\epsilon(m/d))\mu(d),\end{equation}
where $\epsilon(s)=(-1)^{s+1}$.
Therefore, there are many polynomials $f\in \F_q[x]$ of degree $m$ such that $f^{Q_c}$  is irreducible and has degree $Dm$. In the following subsection, we provide background material for the study of irreducible polynomials arising from the rational functions $Q_c$.

\subsection{Polynomials over finite fields}
\medskip

\begin{definition}
For $\alpha\in \overline{\F}_{q}$ and a positive integer $s$, we define the degree of $\alpha$ over $\F_{q^s}$ as the degree $d$ of the minimal polynomial of $\alpha$ over $\F_{q^s}$. We write $d=\deg_{q^s}(\alpha)$.
\end{definition}
The minimal polynomial is always an irreducible polynomial over $\F_{q^s}$ and, in particular, we may define $\deg_{q^s}(\alpha)$ as the minimal positive integer $d$ such that $\alpha\in \F_{q^{ds}}$. For instance, the elements of degree one over $\F_q$ are exactly the elements of $\F_q$: for $a\in \F_q$, its minimal polynomial is $x-a$. Of course, the degree of an element depends on the base field that we are working: an element $\alpha\in \F_{q^2}\setminus \F_q$ satisfies $\deg_q(\alpha)=2$ and $\deg_{q^2}(\alpha)=1$. The following classical result gives a criteria for when a polynomial is irreducible.

\begin{lemma}\label{degree}
Let $g \in \F_{q^d}[x]$ be a polynomial of degree $n$ and $\alpha\in \overline{\F}_q$ be any of its roots. Then $g$ is irreducible over $\F_{q^d}$ if and only if $\deg_{q^d}(\alpha)=n$.
\end{lemma} 

The following lemma shows a relation between the degree and the multiplicative order of an element $\alpha\in \overline{\F}_q^*$.

\begin{lemma}\label{lem:deg-order}
Let $\alpha\in \overline{\F}_q^*$ be an element with multiplicative order $e$. Then $\deg_q(\alpha)=\ord_e(q)$.
\end{lemma}

\begin{proof}
Recall that $\deg_q(\alpha)$ is the least positive integer $d$ such that $\alpha\in \F_{q^d}$. The latter is equivalent to $\alpha^{q^d}=\alpha$ and, since $\alpha\ne 0$, the last equality holds if and only if $\alpha^{q^d-1}=1$. From definition, $\alpha^{q^d-1}=1$ if and only if $q^d\equiv 1\pmod e$ and the least positive integer with this property is $d=\ord_e(q)$.
\end{proof}

\subsubsection{Rational transformations}
For a polynomial $f\in \F_q[x]$ of degree $n$ and a rational function $Q=g/h\in \F_q(x)$, the \emph{$Q$-transform} of $f$ is the following polynomial:
$$f^Q:=h^n\cdot f\left(\frac{g}{h}\right).$$

In the following theorem, we show some basic properties of the $Q$-transforms. Its proof is direct by calculations so we omit the details.

\begin{proposition}\label{prop:Q-transform}
Let $f, g\in \F_q[x]$ and let $Q=g/h\in \F_q(x)$ be a rational function of degree $d$. Suppose that $\deg f=n>1$ and let $a_g$ and $a_h$ be the leading coefficients of $g$ and $h$, respectively. The following hold.
\begin{enumerate}[(a)]
\item If $\deg (g)> \deg (h)$ , $\deg(f^Q)=dn$. 
\item $(f\cdot g)^Q=f^Q\cdot g^Q$. In particular, if $f^Q$ is irreducible over $\F_q$, so is $f$.
\end{enumerate}
\end{proposition}

We observe that $Q_c=\frac{g_c}{h_c}$ with $\deg(g_c)=D>D-1=\deg(h_c)$. In particular, for any polynomial $f\in \F_q[x]$ of degree $n$, $f^{Q_c}$ has degree $nD$, where $D=\deg(Q_c)$ is the order of $[A_c]$ in $\PGL_2(\F_q)$.

\subsubsection{Spin of polynomials}
Following the notation of \cite{mullen}, we introduce the spin of a polynomial and present  some basic results  without  proof. For more details, see Section 4 in  \cite{mullen}.
For a polynomial $f\in \F_{q^d}[x]$ such that $f(x)=x^{m}+a_{m-1}x^{m-1}+\cdots+a_1x+a_0$, we define the following polynomial:
$$f^{(i)}(x)=x^{m}+a_{m-1}^{q^i}x^{m-1}+\cdots+a_1^{q^i}x+a_0^{q^i}.$$
Note that $f^{(s)}=f$ if $s$ is the least common multiple of the numbers $\deg_{q^d}(a_i)$. We set $s(f)=s$.

\begin{definition}
The spin of $f\in \F_{q^d}[x]$ is defined as the following polynomial:
$$S_f(x)=\prod_{i=0}^{s(f)-1}f^{(i)}(x).$$
\end{definition}

\begin{remark}
If $f\in \F_{q^d}[x]$ is not monic and $g\in \F_{q^d}[x]$ is the unique monic polynomial such that $f=\lambda g$ for some $\lambda\in \F_{q^d}^*$, we set $S_f(x):=\lambda^{\frac{q^{ds(g)}-1}{q^d-1}} \cdot S_g(x)$. This is the natural extension of the spin to polynomials that are not monic.
\end{remark}

The following lemma provides a way of computing the spin of a special class of polynomials. Its proof follows directly from calculations so we omit them.
\begin{lemma}[Lemma 11, \cite{mullen}]\label{spin:1}
For polynomials $f, g\in \F_{q}[x]$ and $\lambda$ such that $\deg_q(\lambda)=d$, the spin of the polynomial $f(x)-\lambda g(x)$ equals $$\prod_{i=1}^{d}(f(x)-\lambda^{q^i}g(x)).$$
\end{lemma}




Using the concept of spins, we may derive the factorization of polynomials through rational transformations.

\begin{proposition}[Lemma 13, \cite{mullen}]\label{spin:3}
Let $Q=\frac{g}{h}$ be a rational function over $\F_q$ and let $f\in \F_{q}[x]$ be an irreducible polynomial of degree $d$. If $\alpha\in \F_{q^d}$ is any root of $f$, the factorization of $f^Q$ over $\F_q$ is given by 
$$\prod_{R}S_R(x),$$
where $R$ runs through the irreducible factors of $g-\alpha h$ over $\F_{q^d}$. 
\end{proposition}

The following corollary is a direct consequence of the previous proposition.

\begin{corollary}\label{cor:main2}
Let $Q=\frac{g}{h}$ be a rational function of degree $d$ over $\F_q$ and let $f\in \F_{q}[x]$ be an irreducible polynomial of degree $n$. If $\alpha\in \F_{q^n}$ is any root of $f$, the following hold: 

\begin{enumerate}[(i)]
\item if $g-\alpha h$ is irreducible over $\F_{q^n}$, then $f^Q$ is an irreducible polynomial of degree $dn$ over $\F_q$.
\item if $g-\alpha h$ factors as $\frac{d}{d_0}$ irreducible factors of degree $d_0$ over $\F_{q^n}$, then $f^{Q}$ factors as a product of $\frac{d}{d_0}$ irreducible polynomials, each of degree $d_0n$.
\end{enumerate}
\end{corollary}

\section{Construction of invariant polynomials}
In this section, we discuss the construction of irreducible polynomials via $Q_c$-transforms.  Here we always consider $A_c=\left(\begin{matrix}
0&1\\
c&1\end{matrix}\right)$ such that $\ord([A_c])=D>2$ divides $q+1$ and $Q_c$ is the degree-$D$ canonical rational function associated to $A_c$. We also fix $\theta$ and $\theta^q$, the eigenvalues of $A_c$. Recall that the elements of $C_{A_c}(Dn)$ arise from the $Q_c$-transform of irreducible polynomials of degree $n$. We provide some methods for generating these polynomials via $Q_c$-transforms. Our construction relies in solutions to the following problem:
\begin{enumerate}[{\bf P1}]
\item Given $n>2$ and $Q_c$, find an irreducible polynomial $f$ of degree $n$ such that $f^{Q_c}$ is also irreducible. Here, $f^{Q_c}$ has degree $Dn$.
\end{enumerate}


We provide two solutions to P1: a deterministic method that works when $D$ is a prime and a random method that works for arbitrary $D$. 

\subsection{A recursive method}
Corollary~\ref{cor:main2} motivates us to introduce the following definition.

\begin{definition}
A rational function $Q=\frac{g}{h}\in \F_q(x)$ of degree $d$ is $k$-EDF if, for any positive integer $n>k$ and any $\alpha\in \F_{q^n}$ with $\deg_q(\alpha)=n$, the polynomial $g-\alpha h \in \F_{q^n}[x]$ splits into $\frac{d}{d_0}$ irreducible factors, each of degree $d_0$, for some divisor $d_0$ of $d$. 
\end{definition}

For instance, if $p$ is the characteristic of $\F_q$ and $a\in \F_{q^n}$, it is well known that the polynomial $x^p-x-a$ is either irreducible or splits completely into linear factors over $\F_{q^n}$, hence $Q(x)=x^p-x$ is $0$-EDF. In the following proposition, we present a general family of $2$-EDF rational functions.

\begin{proposition}\label{EDF}
Let $D>2$ be any divisor of $q+1$, $[A_c]\in \PGL_2(\F_q)$ an element of order $D$ and $Q_c$ the canonical rational function associated to $A_c$. Then $Q_c$ is $2$-EDF. In addition, for any $\alpha\in \overline{\F}_q$ with $\deg_q(\alpha)\ge 3$, the polynomial $g_c-\alpha h_c$ is separable. 
\end{proposition}

\begin{proof}
For any $\beta \in S=\overline{\F}_q\setminus \F_{q}$ and $[A]\in \PGL_2(\F_q)$ with $A=\left(\begin{matrix}
a&b\\
c&d\end{matrix}\right)$, we set 
$$[A]\circ \beta=\frac{d\beta-c}{a-b\beta}.$$

By Lemma 2.5 of \cite{ST}, we know that $[A]\circ \beta$ is well defined and $[A]\circ ([B]\circ \beta)=[AB]\circ \beta$ for any $[A], [B] \in \PGL_2(\F_q)$.  Hence 
$$[A^i]\circ \beta:=\underbrace{[A]\circ \cdots [A]\circ}_{i\,\mathrm{times}} \beta.$$
We observe that $[A_c]\circ \beta=\frac{c-\beta}{\beta}$. From Theorem~\ref{main}, $Q_c(x)=Q_c\left(\frac{c}{x+1}\right)$, and so $$Q_c([A_c]\circ \beta)=Q_c\left(\frac{c-\beta}{\beta}\right)=Q_c(\beta),$$ for any $\beta \in S$. Let $\alpha\in \overline{\F}_{q}$ such that $\deg_q(\alpha)=n\ge 3$. Following previous notation, we write $Q_c(x)=\frac{g_c(x)}{h_c(x)}$. Set $F_{\alpha}(x)=g_c(x) -\alpha h_c(x)$ and let $\gamma$ be any root of $F_{\alpha}(x)$, that is, $g_c(\gamma)=\alpha h_c(\gamma)$. Set $d_0=\deg_{q^n}(\gamma)$. Then $F_{\alpha}$ has an irreducible polynomial of degree $d_0$ over $\F_{q^n}$.  Next we show that any other irreducible factor of $F_{\alpha}$ is of the same degree.  Since $h_c$ and $g_c$ are relatively prime and $\alpha \ne 0$, it follows that $Q_c(\gamma)=\alpha$. Since $\deg_q(\alpha)\ge 3$, it follows that $\deg_q(\gamma)\ge 3$ and so $\gamma \in S$. If we set $\gamma_i=[A_c]^{i}\circ \gamma$ for $1\le i\le D$, we have $Q_c(\gamma_i)=\alpha$ and then each $\gamma_i$ is also a root of $F_{\alpha}$. We claim that such elements are pairwise distinct. In fact, if $\gamma_i=\gamma_j$ for some $1\le j<i\le D$, then $$\gamma=\gamma_D=\gamma_{i-j}=[A_c]^{i-j}\circ \gamma,$$ and this yields a nontrivial polynomial  equality in $\gamma$ of degree at most $2$ with coefficients in $\F_q$, i.e., $\deg_q(\gamma)\le 2$, a contradiction. Since $\deg(F_{\alpha})=D$, it  follows that the elements $\gamma_i$ are the roots of $F_{\alpha}$. It is straightforward to check that $\deg_{q^n}(\gamma_i)=\deg_{q^n}(\gamma)$ and this shows that $d_0$ must be a divisor of $D$ and $F_{\alpha}(x)$ splits into $\frac{D}{d_0}$ irreducible factors, each of degree $d_0$.
\end{proof}

Proposition~\ref{EDF} and Corollary~\ref{cor:main2} immediately give
the following corollary.

\begin{corollary}\label{edf-prime}
Suppose that $r>2$ is any prime divisor of $q+1$, $[A_c]\in \PGL_2(\F_q)$ is an element of order $r$ and $Q_c$ is the canonical rational function associated to $A_c$. If $f$ is any irreducible polynomial of degree $n\ge 3$, then $f^{Q_c}$ is either irreducible of degree $rn$ or splits into $r$ distinct irreducible factors over $\F_{q}$, each of degree $n$, depending on  whether $g_c-\alpha h_c\in \F_{q^n}[x]$ is irreducible or splits into linear factors, respectively, where $\alpha\in \F_{q^n}$ is any root of $f$.
\end{corollary}

We next show that $Q$-transform preserves relatively prime relations between polynomials.

\begin{lemma}\label{coprime}
Let $Q=\frac{g}{h}\in \F_q$ be any rational function, where $g$ and $h$ are relatively prime polynomials in $\F_q[x]$. If $f$ and $f_0$ are relatively prime polynomials over $\F_q$, then so are $f^Q$ and $f_0^Q$.\end{lemma}

\begin{proof}Suppose, by contradiction, that there exist relatively prime polynomials $f$ and $f_0$ of degree $n$ such that $f^Q$ and $f_0^Q$ are not relatively prime.  
By Proposition~\ref{spin:3}, it follows that $f$ (resp. $f_0$) has a root $\alpha$ (resp. $\beta$) such that $g-\alpha h$ and $g-\beta h$ are not relatively prime. Since $f$ and $f_0$ are relatively prime, $\alpha\ne \beta$ and, in particular, at least one of these elements is nonzero. Suppose that $\beta\ne 0$. In this case, $\gcd(g-\alpha h, g-\beta h)=\gcd(g-\alpha h, (1-\alpha^{-1}\beta) g)=1$, since $g$ and $h$ are relatively prime and $\alpha\ne \beta$. Thus we have a contradiction.\end{proof}

We fix a prime factor $D>2$ of $q+1$ and $A_c=\left(\begin{matrix}
0&1\\
c&1\end{matrix}\right)$ such that $\ord([A_c])=D$ and $Q_c$ is the degree-$D$ canonical rational function associated to $A_c$. If $f\in \F_q[x]$ is an irreducible polynomial of degree $n\ge 3$, from Corollary~\ref{edf-prime}, the polynomial $f^{Q_c}$ is either irreducible or splits into $D$ irreducible factors, each of degree $n$.  Given such a polynomial $f$, we consider the following procedure:
\begin{enumerate}[]
\item {\bf Step 1} If $f^{Q_c}$ is irreducible, stop. The polynomial $f^{Q_c}$ has degree $nD$.
\item {\bf Step 2} If $f^{Q_c}$ is reducible, each of its irreducible factors has degree $n$, by Corollary~\ref{edf-prime}.
Pick $g$, an irreducible factor of $f^{Q_c}$, and apply {\bf Step 1} again.
\end{enumerate}

We want to guarantee that this process eventually stops, that is, after a finite number of iterations of this procedure, we obtain an irreducible polynomial of degree $nD$ of the form $g^{Q_c}$. In this context, the following definitions are useful.

\begin{definition}
\begin{enumerate}
\item For $f\in \F_q[x]$, set $Q_c^{(0)}\circ f=f$ and, for $i\ge 1$, $Q_c^{(i)}\circ f=(Q_c^{(i-1)}\circ f)^{Q_c}$. 

\item For each positive integer $n\ge 3$, let $\G_n(Q_c)$ be the directed graph with nodes labeled with the monic irreducible polynomials of degree $n$ over $\F_q$ such that the directed edge $f\to g$ is in $\G_n(Q_c)$ if and only if $f$ divides $g^{Q_c}$. 

\item An element $f\in \G_n(Q_c)$ is $Q_c$-periodic if $f$ divides $Q_c^{(k)}\circ f$ for some positive integer $k$, i.e., the node associated to $f$ belongs to a cyclic subgraph of $\G_n(Q_c)$. 
\end{enumerate}
\end{definition}

From Lemma~\ref{coprime}, the following result is straightforward.
\begin{lemma}\label{lem:graph-one}
For any positive integer $n$ and any monic irreducible polynomial $f\in \F_q[x]$ of degree $n$, there exsts at most one polynomial $g\in \F_q[x]$ such that the directed edge $f\to g$ is in $\G_n(Q_c)$. 
\end{lemma}

\begin{remark}\label{remark:iterates-fail}We observe that by Lemma~\ref{lem:graph-one}, the graph is formed by components with
a cycle and trees hanging from the cycle nodes; arrows go towards the
cycle. Moreover, after $i$ iterations of {\bf Step 1} and {\bf Step 2}, if no irreducible of degree $nD$ is obtained, we arrive to a sequence $\{g_j\}_{0\le j\le i}$ of irreducible polynomials of degree $n$ such that $g_0=f$ and $g_j$ divides $g_{j-1}^{Q_c}$ for $1\le j\le i$. It is straightforward to check that, if we arrive at one irreducible polynomial $g_j$ that is not $Q_c$-periodic, our procedure eventually stops. This follows from the fact that the graph $\G_n(Q_c)$ is finite.
\end{remark}

The following theorem provides some general results on the structure of $Q_c$-periodic elements. 
\begin{theorem}\label{thm:periodic-points}
For a monic irreducible polynomial $f\in \F_q[x]$ of degree $n\ge 1$, the following hold:

\begin{enumerate}[(i)]
\item if $f^{Q_c}$ is reducible, at most one of its irreducible factors is $Q_c$-periodic;
\item if $f$ is not $Q_c$-periodic and $f^{Q_c}$ is reducible, none of its irreducible factors is $Q_c$-periodic.
\end{enumerate}
\end{theorem}
\begin{proof}
\begin{enumerate}
\item Suppose that $f^{Q_c}$ is reducible and has two irreducible factors $F_0, F_1$. In particular, the directed edges $F_0\to f$ and $F_1\to f$ are in $\G_n(Q_c)$. From Lemma~\ref{coprime}, for any $g\in \I_n$, there exists at most one element $h\in \I_n$ such that the directed edge $g\to h$ belongs to $\G_n(Q_c)$. From this fact, we can easily see that $F_0$ and $F_1$ cannot be both $Q_c$-periodic.
\item Suppose that $f^{Q_c}$ is reducible and let $g$ be any of its irreducible factors, so $g\rightarrow f$ is an edge of $\G_n(Q_c)$. If $g$ is $Q_c$-periodic, using the fact that there exists at most one element $h\in \I_n$ such that the directed edge $h\to f$ belongs to $\G_n(Q_c)$, we conclude that $f$ is $Q_c$-periodic. This contradicts our hypothesis. 
\end{enumerate}
\end{proof}

From the previous theorem and Remark~\ref{remark:iterates-fail}, we have that if $f$ is any irreducible polynomial of degree $n\ge 3$ such that $f^{Q_c}$ is reducible and we pick $g_0$, $h_0$ two irreducible factors of $f^{Q_c}$, then after a finite number of iterations of {\bf Step 1} and {\bf Step 2} with the polynomials $g_0$ and $h_0$, at least one of them yields an irreducible polynomial of degree $nD$ of the form $G^{Q_c}$ with $G\in \F_q[x]$. In other words, at most one irreducible factor of $f^{Q_c}$ can go to an endless loop when proceeding to {\bf Step 1} and {\bf Step 2}. As follows, we provide upper bounds on the number of iterations of these steps in order to arrive at an irreducible of degree $nD$.

\subsubsection{On the number of iterations}
Recall that $A_c=\left(\begin{matrix}
0&1\\
c&1\end{matrix}\right)$ and $D=\ord([A_c])$. Set $S:=\overline{\F}_q\setminus \F_{q^2}$ and let $P_c:S\to S$ be the map defined by $$P_c(\alpha):=\frac{\alpha+\theta^q}{\alpha+\theta},$$
where $\theta$ is an eigenvalue of $A_c$. In the following lemma, we show some basic properties of the function $P_c$ and how this function interacts with $Q_c$. 

\begin{lemma}\label{lem:aux-dynamic}
The following hold:
\begin{enumerate}[(i)]
\item $P_c$ is well defined and is a permutation of the set $S$.
\item For any $\alpha\in S$,
$$Q_c(\alpha)=P_c^{-1}\circ \Gamma_D\circ P_c(\alpha),$$
where $\Gamma_D: \overline{\F}_q\to \overline{\F}_q$ is given by $\Gamma_D(\alpha)=\alpha^D$. 
\item If $\alpha_1, \ldots, \alpha_m$ are in $S$ and $Q_c(\alpha_i)=\alpha_{i-1}$ for $2\le i\le m$, then 
\begin{equation}\label{conjugation}P_c(\alpha_1)=P_c(\alpha_m)^{D^{m-1}}.\end{equation}
\end{enumerate}
\end{lemma}
\begin{proof}
\begin{enumerate}[(i)]
\item Since for $a\in S$, $\alpha\ne -\theta$, we have that $P_c(\alpha)$ is well defined. We observe that, for $\alpha\in S$,  $P_c(\alpha)\in S$. For this, if $P_c(\alpha)\in \F_{q^2}$, then $P_c(\alpha)^{q^2}=P_c(\alpha)$ and, since $\theta^{q^2}=\theta$, we have that $\alpha^{q^2}=\alpha$, that is, $\alpha\in \F_{q^2}$, a contradiction. It is straightforward to check that the map $T_c:S\to S$ given by $T_c(\alpha)=\frac{\theta\alpha-\theta^q}{1-\alpha}$ is the inverse of $P_c$.
\item From Theorem~\ref{main}, $Q_c$ is well defined in $S$ and so the identity $Q_c(\alpha)=P_c^{-1}\circ \Gamma_D\circ P_c(\alpha)$ follows from the definition of $Q_c$.
\item Since each $\alpha_i$ is in $S$, the equality $Q_c(\alpha_i)=\alpha_{i-1}$ yields $P_c(\alpha_i)^{D}=P_c(Q_c(\alpha_i))=P_c(\alpha_{i-1})$ and we easily obtain Eq.~\eqref{conjugation}.
\end{enumerate}
\end{proof}
 Though our main result here is restricted to the case where $D$ is prime, Lemma~\ref{lem:aux-dynamic} holds for arbitrary $D$. 
 
 For a positive integer $n$, we define $e(n)=\mathrm{lcm}(n, 2)$, i.e., $e(n)=2n$ if $n$ is odd and $e(n)=n$ if $n$ is even. Clearly, if $\alpha\in \F_{q^n}\setminus{\F_{q^2}}$, then $P_c(\alpha)\in \F_{q^{e(n)}}$. Moreover, $D$ also divides $q^{e(n)}-1$ because $D$ divides $q+1$. We obtain the following result.

\begin{theorem}
Let $f\in \F_q[x]$ be an irreducible polynomial of degree $n\ge 3$ and let $\alpha\in \F_{q^n}$ be any of its roots. The following hold:

\begin{enumerate}[(i)]
\item $f$ is $Q_c$-periodic if and only if $\ord(P_c(\alpha))$ is relatively prime with $D$;

\item if $f^{Q_c}$ is reducible, then $\ord(P_c(\alpha))$ divides $\frac{q^{e(n)}-1}{D}$;

\item if $\ord(P_c(\alpha))$ is divisible by $D$ and $\gamma\in \overline{\F}_q$ is any root of $f^{Q_c}$, then $\ord(P_c(\gamma))=D\cdot \ord(P_c(\alpha))$.
\end{enumerate}
In particular, if $f^{Q_c}$ is reducible over $\F_q$ and $g_0, h_0\in \F_q[x]$ are two of its irreducible factors, then after at most
$$\nu_D(q^{e(n)}-1)-1$$
iterations of {\bf Step 1} and {\bf Step 2} with inputs $g_0$ and $h_0$, we obtain an irreducible polynomial of degree $nD$ of the form $G^{Q_c}$.
\end{theorem}
 
\begin{proof}
\begin{enumerate}[(i)]
\item Recall that $f$ is periodic if and only if $f$ divides $Q_c^{(k)}\circ f$ for some $k\ge 1$. The latter holds if and only if $Q_c^{(k)}(\alpha)=\alpha^{q^i}$ for some $i\ge 1$ (or, equivalently, $1\le i\le n$), where $Q_c^{(k)}(\alpha)$ is just the canonical iteration of the function $Q_c$ on $\alpha$, $k$ times. Since $Q_c\in \F_{q}[x]$, $Q_c(\alpha^{q^i})=Q_c(\alpha)^{q^i}$ for $i\in \mathbb N$ and this shows that $Q_c^{(k)}(\alpha)=\alpha^{q^i}$ for some $k, i\ge 1$ if and only if $Q_c^{(k_0)}(\alpha)=\alpha$ for some $k_0\ge 1$. 
From Eq.~\eqref{conjugation}, the last equality is equivalent to
$$P_c(\alpha)=P_c(\alpha)^{D^{k_0}},$$
hence $P_c(\alpha)^{D^{k_0}-1}=1$. If we set $b=\ord(P_c(\alpha))$, then the last equation holds for some $k_0\ge 1$ if and only if $D$ and $b$ are relatively prime.
\item From Corollary~\ref{edf-prime}, if $f^{Q_c}$ is reducible, then $g_c(x)-\alpha h_c(x)$ splits into linear factors over $\F_{q^n}$. In particular, there exists $\gamma\in \F_{q^n}$ such that $g_c(\gamma)=\alpha h_c(\gamma)$. Since $n\ge 3$, $\alpha\ne 0$ and, since $g_c$ and $h_c$ are relatively prime, we have that $g_c(\gamma), h_c(\gamma)\ne 0$. This shows that $Q_c(\gamma)=\alpha$ and, from Eq.~\eqref{conjugation},
$$P_c(\gamma)^{D}=P_c(\alpha),$$ where $P_c(\gamma), P_c(\alpha)\in \F_{q^{e(n)}}$. Since $D$ divides $q^{e(n)}-1$, the last equality implies that $P_c(\alpha)^{\frac{q^{e(n)}-1}{D}}=1$. 

\item Suppose that $\ord(P_c(\alpha))=\ell \cdot D^m$, where $m\ge 1$ and set $\mu=\ord(P_c(\gamma))$. From equation $P_c(\gamma)^{D}=P_c(\alpha)$, we conclude that $\mu$ divides $\ell \cdot D^{m+1}$. If $\mu$ is a proper divisor of $\ell\cdot D^{m+1}$, then there exist a divisor $\ell_0<\ell$ of $\ell$ such that $\mu$ divides $\ell \cdot D^m$ or $\ell_0\cdot D^{m+1}$. We can easily check that both cases contradict to  the fact that  $\ord(P_c(\alpha))=\ell \cdot D^m$. Therefore, $\mu=\ell\cdot D^{m+1}$.
\end{enumerate}
From Theorem~\ref{thm:periodic-points}, either $g_0$ or $h_0$ is not $Q_c$-periodic. Suppose that $g_0$ is not $Q_c$-periodic and, after $i$ iterations of {\bf Step 1} and {\bf Step 2} with input $g_0$, we do not have an irreducible polynomial of degree $nD$. From Remark~\ref{remark:iterates-fail}, this yields a sequence $\{g_j\}_{0\le j\le i}$ of irreducible polynomials of degree $n$ such that $g_j$ divides $g_{j-1}^{Q_c}$ for $1\le j\le i$. The latter yields a sequence of elements $\alpha_j\in \F_{q^n}$ (not uniquely determined) such that $\alpha_j$ is a root of $g_j$ for $1\le j\le i$. In particular, $\alpha_j$ is a root of $g_{j-1}^{Q_c}$. Since $g_0$ is not $Q_c$-periodic and $\alpha_0$ is a root of $g_0$, from item (i), $\ord(P(\alpha_0))$ is divisible by $D$, i.e., $\ord(P_c(\alpha_0))=a\cdot D$ for some integer $a$. A successive application of item (iii) implies that $\ord(P_c(\alpha_i))=a\cdot D^{i+1}$. Since $\alpha_i\in \F_{q^n}$, we have that $P_c(\alpha_i)\in \F_{q^{e(n)}}$ and so $a\cdot D^{i+1}$ divides $q^{e(n)}-1$. In particular, $i+1\le \nu_D(q^{e(n)}-1)$.
\end{proof}

\begin{remark}\label{remark:loop}
Following the proof of the previous theorem, we have that the upper bound $$\nu_D(q^{e(n)}-1)-1$$ on the number
of iterations of {\bf Step 1} and {\bf Step 2} is reached exactly when the initial polynomial $f$ is $Q_c$-periodic. In addition, if $f$ is not $Q_c$-periodic and $f^{Q_c}$ is reducible, we can proceed in the iterations with just one irreducible factor of $f^{Q_c}$, instead of two. In fact, according to Theorem~\ref{thm:periodic-points}, none of the irreducible factors of $f^{Q_c}$ is $Q_c$-periodic and so none of these factors can go to an infinite loop when applying {\bf Step 1} and {\bf Step 2}.\end{remark}

\begin{example}\label{example:type4-recursive}
Suppose that $q=2$ and consider 
$A_1=\left(\begin{matrix}
0&1\\
1&1\end{matrix}\right)$. The element $[A_1]\in \PGL_2(\F_2)$ has order $3$. If $\theta, \theta^2$ are the eigenvalues of $A_1$, $\theta$ is a primitive element of $\F_{2^2}=\F_{4}$ and a direct calculation yields
$$Q_1(x)=\frac{\theta(x+\theta^2)^3-\theta^2(x+\theta)^3}{(x+\theta)^3-(x+\theta^2)^3}=\frac{x^3+x+1}{x^2+x}.$$
For $f(x)=x^4+x+1$, the polynomial $f^{Q_1}$ factors as $3$ irreducible polynomials of degree $4$, one of them being $x^4+x^3+1$. For $f_0(x)=x^4+x^3+1$, we verify that
$$f_1(x)=f_0^{Q_1}(x)=x^{12 }+ x^{11 }+ x^{10 }+ x^{9 }+ x^{8 }+ x^{6 }+ x^{4 }+ x + 1,$$
is an irreducible polynomial over $\F_2$.
\end{example}

\subsection{A random method}
Let $D>2$ be any divisor of $q+1$ and $[A_c]\in \PGL_2(\F_q)$ such that $\ord([A_c])=D$. We recall that, according to Theorem~\ref{main}, any element $f$ of $C_{A_c}(Dn)$ is of the form $g^{Q_c}$ for some polynomial $g$ of degree $n$; from Proposition~\ref{prop:Q-transform}, $g$ must be irreducible. If we suppose $g$ monic, $g^{Q_c}$ may not be monic. For an irreducible polynomial $g\in \F_{q}[x]$ of degree $n$, we define $M(g^{Q_c})$ as the unique monic polynomial of degree $Dn$ equals $g^{Q_c}$ times a constant. In other words, Theorem~\ref{main} entails that for exactly $|C_{A_c}(Dn)|$ elements $g$ of $\I_n$, $M(g^{Q_c})$ is a monic irreducible polynomial of degree $Dn$. If we pick $g\in \I_n$ randomly, what is the probability $p_c(n)$ that $g^{Q_c}$ is irreducible? Of course, we have the following equality:
\begin{equation}\label{probability}p_c(n)=\frac{|C_{A_c}(Dn)|}{|\I_n|}.\end{equation}
Taking estimates to Eq.~\eqref{enumeration-reis}, we can easily obtain the following inequality
$$|C_{A_c}(Dn)|\ge \frac{\varphi(D)}{Dn}\left(q^n-1-\sum_{1\le i\le n/2}(q^{i}+1)\right)\ge \frac{\varphi(D)}{Dn}(q^n-3q^{n/2}).$$
Also, it is well known that $|\I_n|=\frac{1}{n}\sum_{d|n}q^{n/d}\mu(d)$ and so we verify that $|\I_n|\le \frac{q^n}{n}$. In particular, from Eq.~\eqref{probability}, we obtain
$$p_c(n)\ge \frac{\varphi(D)}{D}\cdot \tau(n),$$
where $\tau(n)=1-\frac{3}{q^{n/2}}$. The function $\tau(n)$ goes fast to $1$; for instance, $\tau(n)\ge 0.999$ for $n\ge 24$ and any $q\ge 2$. 

Our procedure is simple: we pick $f\in \I_n$ randomly and check if $f^{Q_c}$ is irreducible. If the answer is yes, we obtain an irreducible polynomial of degree $Dn$ of the form $f^{Q_c}$. If not, we pick another element $f\in \I_n$ and try again.

We note that a random irreducible polynomial of degree $n$ can be obtained in the following way: we pick a random element $\alpha\in \F_{q^n}$ and, with high probability, $\deg_q(\alpha)=n$. Hence the minimal polynomial of $\alpha$ is an irreducible polynomial of degree $n$.

Our procedure is just a trial and error process and, once we succeed, we stop. In particular, our procedure follows the geometric distribution with $p=p_c(n)$; as we know, the expected number of trials is $E(n, c)=\frac{1}{p_c(n)}\le \frac{D}{\varphi(D)\cdot \tau(n)}$. If $\frac{D}{\varphi(D)}$ is small, with high probability, we get success with a small number of steps; for instance, if $D>2$ is any prime number, we have that $E(n, c)\le 2$ for $n\ge 8$ and any $q\ge 2$.

\section{On the transitivity of $Q_c$-transforms}
In this section, we discuss the following problem.

\begin{enumerate}[{\bf P2}]
\item Are the $Q_c$-tranforms transitive (efficient)? That is, given an irreducible polynomial $f$ of degree $n\ge 3$ such that $f^{Q_c}$ is irreducible of degree $nD$, can we guarantee that $Q_c^{(2)}\circ f=(f^{Q_c})^{Q_c}$ is also irreducible?
\end{enumerate}

As we further see,  under a special condition on $n$ and $D$, we prove that if $f^{Q_c}$ is irreducible,  so is $(f^{Q_c})^{Q_c}$.  Hence our construction method is efficient to generate irreducible polynomials of high degree. As we noticed, in Lemma~\ref{lem:aux-dynamic}, the map induced by $Q_c$ on $S=\overline{\F}_q\setminus \F_{q^2}$ is conjugated to the map $x^{D}$ via the permutation $P_c$. Due to this fact, the factorization of $g_c-\alpha h_c$ is related to the factorization of $x^{D}-P_c(\alpha)$. 
We start with the following lemma.

\begin{lemma}\label{binomial}
Let $a\in \overline{\F}_q^{*}$ be an element such that $\ord(a)=e$ and set $k=\ord_e(q)$. For any positive integer $r$ such that $r$ divides $q^k-1$ and any root $b$ of $x^r-a$, we have $\deg_q(b)=\ord_{er}(q)$.
\end{lemma}

\begin{proof}
From Lemma~\ref{lem:deg-order}, $k=\deg_q(a)$. Let $b\in \overline{\F}_q$ be a root of $x^r-a$, $t=\deg_q(b)$ and $t_0=\ord_{er}(q)$, hence $k$ divides $t$. Since $r$ divides $q^k-1$, it follows that $r$ divides $q^t-1$. Therefore, from $b^r=a$ we obtain $$1=b^{q^t-1}=a^{\frac{q^t-1}{r}}.$$ This shows that $\frac{q^t-1}{r}$ is divisible by $e$, hence $q^t\equiv1 \pmod {er}$ and so $t$ is divisible by $t_0$. Let $\delta$ be any primitive element of $\F_{q^{t_0}}$. Since $\ord(a)=e$ divides $q^{t_0}-1$, $a=\delta^{\frac{(q^{t_0}-1)u}{e}}$ for some positive integer $u\le e$ such that $\gcd(u, e)=1$. In particular, $a=\delta_0^{r}$, where $\delta_0=\delta^{\frac{(q^{t_0}-1)u}{er}}$. Moreover, since $q^{t_0}-1$ is divisible by $r$, $\F_{q^{t_0}}$ contains the $r$ elements $\gamma_1, \ldots, \gamma_r$ of $\overline{\F}_q$ such that $\gamma_i^r=1$. The roots of $x^r-a$ are the elements $\gamma_i\delta_0\in \F_{q^{t_0}}$. In particular, $b\in \F_{q^{t_0}}$ and so $\deg(b)=t\le t_0$. Since $t$ is divisible by $t_0$, we have that $t=t_0$, i.e., $\deg_q(b)=\ord_{er}(q)$. 
\end{proof}




The following technical lemmas are useful.

\begin{lemma}\label{degrees}
Let $\alpha \in \overline{\F}_q$ be an element such that $\deg_q(\alpha)=n\ge 3$ and set $P_c(\alpha)=\frac{\alpha+\theta}{\alpha+\theta^q}$, with $\theta, \theta^q$ as before. The following hold:
\begin{enumerate}[(i)]
\item If $n\not\equiv 2\pmod 4$,  then $\deg_q(P_c(\alpha))=e(n)$, where $e(n)=\mathrm{lcm}(n, 2)$.
\item If $n\equiv 2\pmod 4$, then $\deg_q(P_c(\alpha))=n$ or $\frac{n}{2}$.
\end{enumerate}

In particular, if $\deg_q(P_c(\alpha))$ is even, then $\deg_q(P_c(\alpha))=e(n)$.
\end{lemma}

\begin{proof}
Let $d=\deg_q(P_c(\alpha))$. Because $d$ divides $e(d)$,  we have  $P_c(\alpha)^{q^{e(d)}}=P_c(\alpha)$. Since $\theta^{q^2}=\theta$, it follows that $P_c(\alpha^{q^{e(d)}})=P_c(\alpha)$ (recall that $P_c$ is a permutation on $S$ and $\alpha, \alpha^{q^{e(d)}}\in S$). In particular, $\alpha^{q^{e(d)}}=\alpha$ and so $e(d)$ is divisible by $n$.
We observe that, since $\deg(\alpha)=n$, $\alpha\in \F_{q^n}$ and, in particular, $\alpha^{q^n}=\alpha$: since $e(n)$ is even and divisible by $n$, it follows that $P_c(\alpha)^{q^{e(n)}}=P_c(\alpha)$, hence $P_c(\alpha)\in \F_{q^{e(n)}}$. This shows that $e(n)$ is divisible by $d$. In conclusion, $n$ divides $e(d)$ and $d$ divides $e(n)$.

\begin{enumerate}[(i)]
\item Since $n$ divides $e(d)$ and $d$ divides $e(n)$, for $n\equiv 0\pmod 4$ we conclude that $d=n=e(n)$. For $n$ odd, we conclude that $d=n$ or $2n$. If $d=n$ odd, $P_c(\alpha)=P_c(\alpha)^{q^n}=P_c(\alpha)^{-1}$. However, for any $\deg(\alpha)=n\ge 3$, $P_c(\alpha)\not \in \F_{q^2}$ and so $P_c(\alpha)\ne \pm 1$. In particular, $P_c(\alpha)^{-1}\ne P_c(\alpha)$, hence $d=2n=e(n)$ for $n$ odd.

\item Since $n$ divides $e(d)$ and $d$ divides $e(n)$, for $n\equiv 2\pmod 4$, we conclude that $d=n=e(n)$ or $d=\frac{n}{2}$.
\end{enumerate}
Since $e(n)$ is even and $\frac{n}{2}$ is odd for $n\equiv 2\pmod 4$, if $\deg_q(P_c(\alpha))$ is even, $\deg_q(P_c(\alpha))=e(n)$.
\end{proof}

The following lemma is a particular case of the Lifting the Exponent Lemma (LTE), a well-known result in the Olympiad folklore. For its proof, see Proposition 1 of~\cite{B77}.

\begin{lemma} Let $r$ be a prime and $b$ a positive integer such that $r$ divides $b-1$. For any positive integer $n$, the following hold:

\begin{enumerate}[(i)]
\item if $r$ is odd, then $\nu_r(b^n-1)=\nu_r(b-1)+\nu_r(n)$;
\item if $r=2$, then $\nu_2(b^n-1)=\nu_2(b-1)+\nu_2(n)+\nu_2(b+1)-1$ if $n$ is even and $\nu_2(b^n-1)=\nu_2(b-1)$ if $n$ is odd.
\end{enumerate}
\end{lemma}

In particular, from the previous lemma, it follows that for any positive integer $d$ the equality
\begin{equation}\label{LTE}\nu_r(b^d-1)=\nu_r(b-1)+\nu_r(d),\end{equation}
holds if $r$ is an odd prime divisor of $b-1$ or $r=2$ and $b\equiv 1\pmod 4$. From this last equality, we obtain the following result.

\begin{lemma}\label{lifting}
Let $e$ be a positive integer such that $\ord_e(q)=n$ is even and let $r$ be a positive integer such that every prime divisor of $r$ also divides $q^n-1$. Write $q^n-1=e\cdot r_0\cdot E$, where every prime divisor of $r_0$ divides $r$ and $\gcd(E, r)=1$. Then $\ord_{er}(q)=n\cdot \frac{r}{\gcd(r, r_0)}$.
\end{lemma}

\begin{proof}
We observe that $\ord_{er}(q)$ is divisible by $\ord_e(q)=n$. Write $\ord_{er}(q)=ns$, hence $s$ is the least positive integer such that $er$ divides $q^{ns}-1$. Since $q^{ns}-1=\frac{q^{ns}-1}{q^n-1}\cdot (q^n-1)$ and $q^n-1=e\cdot r_0\cdot E$, we conclude that $er$ divides $q^{ns}-1$ if and only if $\frac{r}{\gcd(r, r_0)}$ divides $\frac{q^{ns}-1}{q^n-1}$. If $T$ is any prime divisor of $\frac{r}{\gcd(r, r_0)}$ (hence a divisor of $r$), from hypothesis, $T$ divides $q^n-1$. Since $n$ is even, if $T=2$, then $q$ is odd and so $q^n\equiv 1\pmod 4$. In particular, we can apply Eq.~\eqref{LTE} for $b=q^{n}$ and $d=s$ and so we conclude that $$\nu_T\left(\frac{q^{ns}-1}{q^n-1}\right)=\nu_T(s).$$
From the last equality, we conclude that $er$ divides $q^{ns}-1$ if and only if $s$ is divisible by $\frac{r}{\gcd(r, r_0)}$. Since $s$ is minimal, we have that $s=\frac{r}{\gcd(r, r_0)}$.
\end{proof}

The following theorem shows that, under some generic conditions, the $Q_c$-transform is (transitive) efficient.

\begin{theorem}\label{thm:efficient}
Let $n\ge 3$ be a positive integer, $f\in \F_q[x]$ an irreducible polynomial of degree $n$ and $Q_c$ the canonical rational function associated to $A_c$ with $\ord([A_c])=D$. Let $f^{Q_c}$ be defined as in Definition~\ref{def:Q_c}. Suppose that $f^{Q_c}$ is irreducible. The following hold:

\begin{enumerate}
\item if $D\not\equiv 2\pmod 4$ or $n$ is even, then $(f^{Q_c})^{Q_c}$ is also irreducible,

\item if $D\equiv 2\pmod 4$ and $n$ is odd, then $(f^{Q_c})^{Q_c}$ is either irreducible or factors into two irreducible polynomials of degree $\frac{nD^2}{2}$.
\end{enumerate}
\end{theorem}

\begin{proof}
Let $\alpha$ be a root of $f$, $\beta$ be a root of $g_c-\alpha h_c$ and $\gamma$ be a root of $g_c-\beta h_c$. One can easily see that, under these assumptions, $\beta$ is a root of $f^{Q_c}$ with $Q_c(\beta)=\alpha$ and $\gamma$ is a root of $(f^{Q_c})^{Q_c}$ with $Q_c(\gamma)=\beta$. Since $\beta$ is a root of $f^{Q_c}$, that is an irreducible polynomial of degree $nD$, $\deg(\beta)=nD$. We observe that $(f^{Q_c})^{Q_c}$ has degree $nD^2$ and vanishes at $\gamma$. If we set $g=f^{Q_c}$, we have $g$ irreducible and $g^{Q_c}=(f^{Q_c})^{Q_c}$. From Proposition~\ref{EDF}, there exists a divisor $d$ of $D$ such that $g^{Q_c}$ factors as a product of $D/d$ irreducible polynomials over $\F_q$, each of degree $dDn$. This implies that $\deg(\gamma)=dDn$. In particular, in order to prove our result, we just need to show that $d=D$ if $D\not\equiv 2\pmod 4$ or $n$ is even, and $d=D/2$ or $d=D$ if $D\equiv 2\pmod 4$ and $n$ is odd. From Eq.~\eqref{conjugation}, $P_c(\beta)^D=P_c(\alpha)$ and $P_c(\gamma)^{D^2}=P_c(\alpha)$.

\begin{enumerate}[]
\item {\bf Claim.} \emph{The degrees (over $\F_q$) of the elements $P_c(\alpha), P_c(\beta)$ and $P_c(\gamma)$ are even.} 

\item Indeed, since $P_c(\alpha)$ is a power of the elements $P_c(\beta)$ and $P_c(\gamma)$, the degrees of such elements are divisible by the degree of $P_c(\alpha)$ and so it suffices to prove that $\deg_q(P_c(\alpha))$ is even. If $\deg_q(P_c(\alpha))$ were odd, from Lemma~\ref{degrees}, $n\equiv 2\pmod 4$ and $\deg_q(P_c(\alpha))=\frac{n}{2}$. Since $n\ge 3$, $P_c(\alpha)\ne 0$ and, in particular, $P_c(\alpha)^{q^{n/2}-1}=1$.  From the equality $P_c(\beta)^{D}=P_c(\alpha)$, we conclude that $$P_c(\beta)^{D(q^{n/2}-1)}=P_c(\alpha)^{q^{n/2}-1}=1.$$ Since $n/2$ is odd and $D$ divides $q+1$, $D$ divides $q^{n/2}+1$, hence $D(q^{n/2}-1)$ divides $q^n-1$ and so $P_c(\beta)^{q^n-1}=1$, i.e., $P_c(\beta)\in \F_{q^n}$. Since $n$ is even, $\F_{q^n}$ contains the coefficients of $P_c(x)$ and $P_c^{-1}(x)$. Therefore, $\beta=P_c^{-1}(P_c(\beta))$ is also in $\F_{q^n}$, a contradiction with $\deg_q(\beta)=Dn$. This completes the proof of the claim.
\end{enumerate}
From the previous claim and Lemma~\ref{degrees}, it follows that $$\begin{cases}\deg_q(P_c(\alpha))&=e(n),\\ \deg_q(P_c(\beta))&=e(Dn),\\ \deg_q(P_c(\gamma))&=e(dDn).\end{cases}$$ Set $E=\ord(P_c(\alpha))$, the multiplicative order of $P_c(\alpha)$. Recall that $P_c(\beta)^D=P_c(\alpha)$ and so $P_c(\beta)$ is a root of $x^D-P_c(\alpha)$. The number $e(n)$ is even, so $q^{e(n)}-1$ is divisible by $q+1$, hence is divisible by $D$.  From Lemma~\ref{binomial}, $\deg_q(P_c(\beta))=\ord_{DE}(q)$, hence $\ord_{DE}(q)=e(Dn)$. Since $\deg_q(P_c(\alpha))=e(n)$, it follows that $\ord_E(q)=e(n)$ (see Lemma~\ref{lem:deg-order}). Write $q^{e(n)}-1=E\cdot D_0\cdot s$, where every prime divisor of $D_0$ divides $D$ and $\gcd(s, D)=1$. From Lemma~\ref{lifting}, $\ord_{DE}(q)=e(n)\cdot \frac{D}{\gcd(D, D_0)}$ and so we conclude that $e(n)\cdot \frac{D}{\gcd(D, D_0)}=e(Dn)$, i.e., $\gcd(D, D_0)=\frac{e(n)D}{e(Dn)}$. Analyzing the parity of $D$ and $n$, we have that $\frac{e(n)D}{e(Dn)}=1$ unless $D$ is even and $n$ is odd and, in this case, $\frac{e(n)D}{e(Dn)}=2$. Since $\gcd(D, D_0)=\frac{e(n)D}{e(Dn)}$ and every prime divisor of $D_0$ divides $D$, it follows that $D_0=1$ unless $D$ is even and $n$ is odd and, in this case, $D_0=2^m$ for some $m\ge 1$.
Therefore, $\gcd(D^2, D_0)=\gcd(D, D_0)$ unless $D\equiv 2\pmod 4$, $n$ is odd and $D_0=2^m$ with $m\ge 2$. In the latter case, one has $\gcd(D, D_0)=2$ and $\gcd(D^2, D_0)=4$. We split into cases.

\begin{itemize}
\item {\bf Case 1:} $D\not\equiv 2\pmod 4$ or $n$ is even. In this case, we have that $\gcd(D^2, D_0)=\gcd(D, D_0)$. Since $P_c(\gamma)$ is a root of $x^{D^2}-P_c(\alpha)$ and $e(dDn)=\deg_q(P_c(\gamma))$, from Lemma~\ref{binomial}, $e(dDn)=\ord_{D^2E}(q)$. From Lemma \ref{lifting}, we have that $\ord_{D^2E}(q)=\frac{e(n)D^2}{\gcd(D^2, D_0)}$ and so $$\ord_{D^2E}(q)=\frac{e(n)D^2}{\gcd(D, D_0)}=D\cdot e(Dn).$$ Because $d$ is a divisor of $D$, a simple analysis on the parity of $d$ shows that $e(dDn)=d\cdot e(Dn)$ and so 
$$d\cdot e(Dn)=e(dDn)=\ord_{D^2E}(q)=D\cdot e(Dn).$$
Therefore, $d=D$ and so $(f^{Q_c})^{Q_c}$ is irreducible.
\item {\bf Case 2:} $D\equiv 2\pmod 4$ and $n$ is odd. In this case, we have $\gcd(D^2, D_0)=a\cdot \gcd(D, D_0)$, where $a=2$ or $a=1$, according to whether $D_0$ is divisible by $4$ or not, respectively. Following the proof of {\bf Case 1}, we obtain $d=D/a$.
\end{itemize} 
\end{proof}
The following corollary is immediate.
\begin{corollary}\label{cor:transitive}
Let $n\ge 3$ be a positive integer, $f\in \F_q[x]$ an irreducible polynomial of degree $n$ and $Q_c$ the canonical rational function associated to $A_c$ with $\ord([A_c])=D$. Suppose that $f^{Q_c}$ is irreducible. If $D\not\equiv 2\pmod 4$ or $n$ is even, then the sequence $\{f_i\}_{i\ge 0}$ with $f_0=f$ and $f_i=f_{i-1}^{Q_c}$ is a sequence of irreducible polynomials such that $\deg(f_i(x))=nD^i$ for any $i\ge 0$.
\end{corollary}

We conclude this section with two examples.

\begin{example}
Suppose that $q=5$ and consider 
$A_3=\left(\begin{matrix}
0&1\\
3&1\end{matrix}\right)$. The element $[A_3]\in \PGL_2(\F_5)$ has order $6$ and $6\equiv 2\pmod 4$. If $\theta, \theta^5$ are the eigenvalues of $A_3$, $\theta$ is a primitive element of $\F_{5^2}=\F_{25}$ and a direct calculation yields
$$Q_3(x)=\frac{\theta(x+\theta^5)^6-\theta^5(x+\theta)^6}{(x+\theta)^6-(x+\theta^5)^6}=\frac{x^6+x+2}{x^5-x}.$$
For $f(x)=x^3+4x+3$, one can verify that $g=f^{Q_c}$ is an irreducible polynomial of degree $18=3\cdot 6$ but $g^{Q_c}$ factors as two irreducible factors of degree $54=\frac{3\cdot 6^2}{2}$. If we set $f_0(x)=x^6+2x+3$, one can verify that
$$f_1(x)=f_0^{Q_3}(x)=x^{36}+3x^{31} + 4x^{26} + 4x^{25} + 4x^{11} + 2x^{10} + 4x^6 + 3x^5 + 2x + 4,$$
is an irreducible polynomial. In particular, from Corollary~\ref{cor:transitive}, the sequence $f_i$ defined recursively by 
$$f_i(x)=(x^5-x)^{6^i}\cdot f_{i-1}\left(\frac{x^6+x+2}{x^5-x}\right),$$
is a sequence of irreducible polynomials of degree $6^{i+1}$ over $\F_5$. The first terms of this sequence are given as follows.
{\small\begin{align*}
f_0(x)=&\,x^6+2x+3,\\
f_1(x)=&\,x^{36}+3x^{31} + 4x^{26} + 4x^{25} + 4x^{11} + 2x^{10} + 4x^6 + 3x^5 + 2x + 4,\\
f_2(x)=&\,x^{216 }+ 4x^{211 }+ 2x^{210 }+ 2x^{206 }+ 3x^{200 }+ 3x^{191 }+ 2x^{190 }+ 4x^{186 }+ x^{185 }+ x^{181 }+\\{}& 4x^{176 }+ x^{175 }+ x^{166 }+ 3x^{156 }+4x^{155 }+ x^{151 }+ 4x^{150 }+ 2x^{141 }+ 3x^{140 }+ 2x^{136 }+\\{}& x^{135 }+ 2x^{131 }+ 4x^{130 }+ 3x^{126 }+2x^{125 }+ 2x^{91 }+ 3x^{90 }+ 3x^{86 }+  x^{85 }+ 4x^{81 }+ \\{}&3x^{80 }+ 2x^{76 }+ 4x^{75 }+ 2x^{66 }+ 4x^{65 }+  2x^{61 }+ 3x^{60 }+ x^{56 }+ 4x^{51 }+ 2x^{50 }+ 3x^{41 }+\\{}& x^{40 }+ 3x^{35 }+ x^{31 }+   3x^{30 }+ 3x^{26 }+ 4x^{16 }+ 3x^{15 }+ 2x^{11 }+ 3x^{10 }+ 4x^{6 }+ 2x^{5 }+\\{}& 3x +1.
\end{align*}}
\end{example}

\begin{example}
Suppose that $q=2$ and consider 
$A_1=\left(\begin{matrix}
0&1\\
1&1\end{matrix}\right)$. The element $[A_3]\in \PGL_2(\F_2)$ has order $3$. If $\theta, \theta^2$ are the eigenvalues of $A_3$, $\theta$ is a primitive element of $\F_{2^2}=\F_{4}$ and a direct calculation yields
$$Q_1(x)=\frac{\theta(x+\theta^2)^3-\theta^2(x+\theta)^3}{(x+\theta)^3-(x+\theta^2)^3}=\frac{x^3+x+1}{x^2+x}.$$
From Example~\ref{example:type4-recursive}, $f_0(x)=x^4+x^3+1$ and $$f_1(x)=f_0^{Q_1}(x)=x^{12 }+ x^{11 }+ x^{10 }+ x^{9 }+ x^{8 }+ x^{6 }+ x^{4 }+ x + 1,$$
are irreducible over $\F_2$. In particular, from Corollary~\ref{cor:transitive}, the sequence $f_i$ defined recursively by 
$$f_i(x)=(x^2+x)^{4\cdot 3^{i-1}}\cdot f_{i-1}\left(\frac{x^3+x+1}{x^2+x}\right),$$
is a sequence of irreducible polynomials of degree $4\cdot 3^i$ over $\F_2$. The first terms of this sequence are given as follows.
{\small\begin{align*}
f_0(x)=&\,x^4+x^3+1,\\
f_1(x)=&\,x^{12 }+ x^{11 }+ x^{10 }+ x^{9 }+ x^{8 }+ x^{6 }+ x^{4 }+ x + 1,\\
f_2(x)=& \,x^{36 }+ x^{35 }+ x^{32 }+ x^{28 }+ x^{27 }+ x^{26 }+ x^{20 }+ x^{17 }+ x^{12 }+ x^{11 }+ x^{9 }+ x^{8 }+\\{}& x^{4 }+ x^{3 }+ x^{2 }+ x + 1,\\
f_3(x)=& \,x^{108 }+ x^{107 }+ x^{106 }+ x^{105 }+ x^{104 }+ x^{102 }+ x^{99 }+ x^{91 }+ x^{90 }+ x^{89 }+ x^{88 }+ x^{86 }+ \\{}&x^{84 }+ x^{83 }+ x^{80 }+ x^{75 }+ x^{73 }+ x^{68 }+  x^{67 }+ x^{66 }+  x^{64 }+ x^{60 }+  x^{58 }+ x^{54 }+ x^{50 }+\\{}& x^{48 }+ x^{44 }+ x^{43 }+ x^{42 }+ x^{40 }+ x^{36 }+ x^{34 }+ x^{33 }+ x^{32 }+ x^{27 }+x^{22 }+ x^{17 }+  x^{12 }+\\{}& x^{11 }+ x^{6 }+ x^{4 }+ x + 1,\\
\end{align*}}
{\small\begin{align*}
f_4(x)=& \,x^{324 }+ x^{323 }+ x^{320 }+ x^{308 }+ x^{307 }+ x^{306 }+ x^{291 }+ x^{290 }+ x^{289 }+ x^{274 }+ x^{273 }+\\{}& x^{272 }+ x^{259 }+ x^{258 }+ x^{257 }+ x^{256 }+  x^{244 }+ x^{243 }+ x^{242 }+ x^{228 }+  x^{227 }+ x^{225 }+\\{}& x^{224 }+ x^{210 }+ x^{209 }+ x^{208 }+ x^{195 }+ x^{192 }+ x^{180 }+ x^{179 }+x^{177 }+ x^{176 }+  x^{164 }+\\{}& x^{161 }+ x^{148 }+ x^{147 }+ x^{146 }+ x^{131 }+  x^{130 }+ x^{129 }+ x^{114 }+ x^{113 }+ x^{112 }+ x^{100 }+\\{}& x^{99 }+x^{98 }+ x^{84 }+ x^{83 }+ x^{81 }+   x^{80 }+ x^{67 }+ x^{50 }+ x^{49 }+ x^{48 }+ x^{35 }+ x^{32 }+ x^{20 }+\\{}& x^{19 }+x^{17 }+ x^{16 }+ x^{4 }+ x^{3 }+ x^{2 }+ x + 1.
\end{align*}}
\end{example}

\section{Future research}
Here we propose a problem and a conjecture based on theoretical and practical considerations. We recall that we provided a recursive and a random method for producing irreducible polynomials of the form $f^{Q_c}$. We ask if we may obtain a criterion for the irreducibility for polynomials $f^{Q_c}$. More specifically, we propose the following problem.

\begin{problem}
Let $f\in \F_q[x]$ be an irreducible polynomial of degree $n\ge 3$. Provide necessary and sufficient conditions on $f$ to ensure the irreducibility of $f^{Q_c}$.
\end{problem}
 
Recall that Theorem~\ref{thm:efficient} entails that $Q_c$ transforms are partially (transitive) efficient; if $f\in \F_q[x]$ is irreducible of degree $n\ge 3$ such that $f^{Q_c}$ is irreducible, then $(f^{Q_c})^{Q_c}$ is also irreducible provided that $n$ is even or $D\not\equiv 2\pmod 4$, where $D$ is the degree of $Q_c$. In addition, if $D\equiv 2\pmod 4$ and $n$ is odd, $(f^{Q_c})^{Q_c}$ is either irreducible or factors into two irreducible polynomials of degree $\frac{nD^2}{2}$. Based on some computational tests, we believe that in the former situation, $(f^{Q_c})^{Q_c}$ is never irreducible.

\begin{conjecture}
Let $f$ be an irreducible polynomial of degree $n\ge 3$ and suppose that $Q_c$ has degree $D$. In addition, suppose that $n$ is odd and $D\equiv 2\pmod 4$. If $f^{Q_c}$ is irreducible, then $(f^{Q_c})^{Q_c}$ factors into two irreducible polynomials of degree $\frac{nD^2}{2}$ over $\F_q$.
\end{conjecture}

\begin{center}
{\bf Acknowledgments}
\end{center}

This work was conducted during a visit of the second author to Carleton University, supported by the Program CAPES-PDSE (process - 88881.134747/2016-01).


\end{document}